\documentclass[a4paper,10pt]{amsart}
\usepackage[shortlabels]{enumitem}
\usepackage{amsthm,amsmath,amssymb,graphics,graphicx,hyperref,epstopdf,mathrsfs}
\usepackage{mathtools}
\usepackage{breqn}      
\usepackage{tikz}  
\usetikzlibrary{shapes.geometric}

\usetikzlibrary{calc,intersections,through}
  \usepackage{capt-of}

\title{Monomial invariants applied to graph coloring}
\author{Guillermo Alesandroni}
\address{2000 Rosario, Santa Fe, Argentina}
\email{guillea@okstate.edu, alesandronig@yahoo.com}

\newtheorem{theorem}{Theorem}[section]
\newtheorem{proposition}[theorem]{Proposition}
\newtheorem{corollary}[theorem]{Corollary}
\newtheorem{lemma}[theorem]{Lemma}

\newtheorem{conjecture}[theorem]{Conjecture}

\theoremstyle{definition}
\newtheorem{definition}[theorem]{Definition}

\newtheorem{example}[theorem]{Example}

\DeclareMathOperator{\pd}{pd}

\DeclareMathOperator{\lcm}{lcm}

\DeclareMathOperator{\pol}{pol}
\DeclareMathOperator{\codim}{codim}

\DeclareMathOperator{\e}{e}

\begin{document}
\maketitle
\begin{abstract}
This article is built upon three main ideas. First, for a class of monomial ideals, it is proven that the multiplicity of an ideal equals the number of realizations of its codimension (an intuitive concept that we define later). Next, for an arbitrary graph $G$, we construct a monomial ideal $M_G$, and show that the chromatic number of $G$ is equal to the codimension of $M_G$. Finally, for a class of graphs, we give a formula that computes the chromatic polynomial of $G$, evaluated at the chromatic number of $G$, in terms of the codimension and multiplicity of $M_G$. In particular, the formula applies to all graphs satisfying the Erdős-Faber-Lovász conjecture.
\end{abstract}
\section{Introduction}
The flow of this work is a threefold process that aims at computing graph invariants, using monomial invariants. Although the ultimate goal is to put the theory of monomial ideals to the service of graph coloring, each step of the process has its own goal, and contains some result interesting in itself. 

In the first stage, we reinterpret two monomial invariants in the language of monomial ideals. Since monomial ideals are members of the larger class of modules over a polynomial ring, definitions and theorems about such modules automatically become definitions and theorems about monomial ideals. In particular, the concepts of monomial codimension and monomial multiplicity can be (and have often been) defined with the same terminology as the codimension and multiplicity of a module. However, in this article we need to redefine these concepts in terms inherent to monomial ideals. 

More specifically, we define the codimension of a monomial ideal $M$ as the cardinality of the smallest set of variables $X$, such that every generator of $M$ is divisible by some variable of $X$. The set $X$ is called a realization of the codimension of $M$. Then, we prove that for a large class of monomial ideals, the multiplicity of an ideal equals the number of realizations of the codimension of its polarization (Theorem \ref{realizations}). 

The second part of this work is constructive and pivotal. Given a graph $G$, a monomial ideal $M_G$ is constructed, and we call it the chromatic ideal of $G$. The name chromatic ideal responds to the fact that  the chromatic number of $G$ equals the codimension of $M_G$ (Theorem \ref{Theorem 3}). This chromatic ideal $M_G$ can be regarded as the algebraic counterpart of the graph $G$, as invariants of one can often be represented in terms of invariants of the other, and each structure can be obtained from the other (Proposition \ref{Prop}).

The idea of associating a monomial ideal to a graph has already been used in the past. The term edge ideal, for instance, clearly speaks of an ideal defined in terms of the edges of a graph. An interesting application of the interplay between monomial ideals and graphs can be found in [FHVT], where the authors use cover ideals to compute chromatic numbers. Edge ideals, cover ideals, and other ideals constructed with an eye on graphs, usually have the following common characteristic: the vertices of the graph determine the variables of the ideal, and the edges of the graph determine (either directly or indirectly) the minimal generators of the ideal. Chromatic ideals, however, reverse this principle. The vertices of $G$ define the minimal generators of $M_G$, and the edges of $G$ determine the variables of $M_G$ (Definition \ref{DefChrom}). This distictive feature makes chromatic ideals a useful bridge between commutative algebra and graph theory.

In the third and last part of the process, we transform information about monomial ideals into information about graphs. Some graph invariants are expressed in terms of monomial invariants and, in our main result (Theorem \ref{Theorem 5.6}), we show how the chromatic polynomial of $G$, evaluated at the chromatic number of $G$, can be computed using the codimension and multiplicity of $M_G$, for a class of graphs. The key to prove this theorem, is the fact that the number of configurations of $k$-colorings of $G$ (where $k$ is the chromatic number of $G$) is equal to the number of realizations of the codimension of $M_G$. In particular, the theorem applies to all graphs satisfying the Erdős-Faber-Lovász conjecture, something that we discuss at the end of this paper.

The material is organized as follows. Section 2 concerns background and notation. Section 3 discusses the meanings of codimension and multiplicity in the context of monomial ideals. In Section 4, we introduce chromatic ideals, compute some of their invariants, and prove their basic properties. Section 5 displays the interaction between chromatic polynomial, chromatic number, codimension, and multiplcity. Section 6 deals with the Erdős-Faber-Lovász conjecture. Finally, in Section 7, we make some remarks and ask some questions.

\section{Background and notation}
Since this article relies more strongly on commutative algebra than it does on graph theory, we will presuppose some familiarity with the former. In particular, we will take for granted that the reader is acquainted with dimension theory and free resolutions, and knows some properties of monomial ideals. On the other hand, we will not assume previous knowledge on graph theory. In fact, we will introduce the fundamental objects, and will define any terms as needed. For related work on commutative algebra and graph theory, see [Pe] and [VT], respectively.

Throughout, the letter $S$ denotes a polynomial ring over a field, in sufficiently many variables. The letter $M$ always represents a monomial ideal in $S$, and the standard notation $G$ is reserved for a finite simple undirected graph (a concept that we define below). All graphs studied in this article are finite, simple, and undirected. For this reason, we will omit these adjectives and will say that $G$ is a graph, without ambiguity.

\begin{definition}
A \textbf{graph $G$} is a pair $G=(V,E)$, where $V$ is a finite set of elements called \textbf{vertices}, and $E$ is a family of $2$-element subsets of $V$, called \textbf{edges}. Unless otherwise stated, the vertices will be denoted with numbers. Thus, a graph will usually be represented in the form $G=(V,E)$, where $V=\{1,\ldots,n\}$, and $E\subseteq\left\{\{i,j\}: (i,j)\in V^2, i\ne j \right\}$.

If $G=(V,E)$ and $G'=(V',E')$ are two graphs, such that $V' \subseteq V$ and $E' \subseteq E$, we say that $G'$ is a \textbf{subgraph} of $G$. In particular, if every edge $\{i,j\}$ of $E$ with $i,j \in V'$, is also an edge of $E'$, then $G'$ is called the \textbf{induced subgraph} of $G$ on $V'$.
\end{definition}

Below, we define some special classes of graphs.

\begin{definition}
Let $G=(V,E)$ be a graph.
\begin{enumerate}[(i)]
\item If $V=\{1,\ldots,n\}$ and $E=\left\{\{i,j\}: (i,j)\in V^2, i\ne j \right\}$, $G$ is called an \textbf{$n$-clique}.
\item If $V=\{1,\ldots,n\}$ and $E=\left\{\{1,2\},\{2,3\},\ldots,\{n-1,n\},\{n,1\}\right\}$, $G$ is called an \textbf{$n$-cycle}.
\item If $E=\varnothing$, $G$ is called an \textbf{edgeless graph}.
\end{enumerate}
\end{definition}

Graphs are usually represented graphically by identifying each vertex with a node, and each edge $\{i,j\}$ with a segment of line joining the nodes identified with the vertices $i$ and $j$. The following graphs illustrate the definition above.

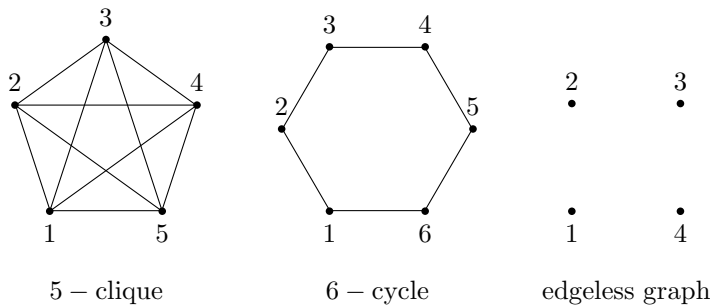
\begin{example}
\[
\begin{array}{ccccc}
\begin{tikzpicture}
\foreach \a in {5}{
\node[regular polygon, regular polygon sides=\a, minimum size=2.5cm, draw]  (A) {};
\foreach \i in {3}
    \node[circle, label=below:{1}, fill=black, inner sep=1pt] at (A.corner \i) {};
    \foreach \i in {2}
    \node[circle, label=above:{2}, fill=black, inner sep=1pt] at (A.corner \i) {};
    \foreach \i in {1}
    \node[circle, label=above:{3}, fill=black, inner sep=1pt] at (A.corner \i) {};
    \foreach \i in {5}
    \node[circle, label=above:{4}, fill=black, inner sep=1pt] at (A.corner \i) {};
    \foreach \i in {4}
    \node[circle, label=below:{5}, fill=black, inner sep=1pt] at (A.corner \i) {};
    \draw (A.corner 1)--(A.corner 3);
    \draw (A.corner 2)--(A.corner 5);
    \draw (A.corner 2)--(A.corner 4);
    \draw (A.corner 1)--(A.corner 4);
    \draw (A.corner 3)--(A.corner 5);
    }
    
  \end{tikzpicture}
&&\begin{tikzpicture}
\foreach \a in {6}{
\node[regular polygon, regular polygon sides=\a, minimum size=2.5cm, draw]  (A) {};
\foreach \i in {4}
    \node[circle, label=below:{1}, fill=black, inner sep=1pt] at (A.corner \i) {};
    \foreach \i in {3}
    \node[circle, label=above:{2}, fill=black, inner sep=1pt] at (A.corner \i) {};
    \foreach \i in {2}
    \node[circle, label=above:{3}, fill=black, inner sep=1pt] at (A.corner \i) {};
    \foreach \i in {1}
    \node[circle, label=above:{4}, fill=black, inner sep=1pt] at (A.corner \i) {};
    \foreach \i in {6}
    \node[circle, label=above:{5}, fill=black, inner sep=1pt] at (A.corner \i) {};
     \foreach \i in {5}
    \node[circle, label=below:{6}, fill=black, inner sep=1pt] at (A.corner \i) {};
}
\end{tikzpicture}&&
\begin{tikzpicture}
\node[draw=none,minimum size=2cm,regular polygon,regular polygon sides=4] (a) {};

\foreach \x in {1,2,...,4}
  \fill (a.corner \x) circle[radius=1pt];
  \node[circle, label=below:{1}, fill=black, inner sep=1pt] at (a.corner 3){};
  \node[circle, label=above:{3}, fill=black, inner sep=1pt] at (a.corner 1){};
  \node[circle, label=above:{2}, fill=black, inner sep=1pt] at (a.corner 2){};
  \node[circle, label=below:{4}, fill=black, inner sep=1pt] at (a.corner 4){};
\end{tikzpicture}\\
5-\text{clique}& &6-\text{cycle}&&\text{edgeless graph} 
\end{array}\]
\end{example}

The following are some common terms derived from the visual representation of a graph. Two vertices connected by an edge are said to be \textbf{adjacent}, and are called \textbf{endpoints} of such edge. Two edges with a common vertex are called \textbf{incident}. In addition, if a vertex is one of the endpoints of an edge, the edge and the vertex are said to be \textbf{incident}.

\begin{definition}
Given a graph $G=(V,E)$, and a set of vertices $\omega\subseteq V$, we say that $\omega$ is an \textbf{independent set} if no edge has both endpoints in $\omega$. The set $\omega$ is called a \textbf{maximal independent set} if it is maximal under inclusion. 
\end{definition}

We close this section with one final important definition.

\begin{definition}
A \textbf{$k$-coloring} of a graph $G = (V,E)$ is a function $f:V\rightarrow \{0,1,\ldots,k-1\}$, such that $f(i)\ne f(j)$ whenever $i$ is adjacent to $j$. For a given value of $k$, we say that $G$ is \textbf{$k$-colorable}, if a $k$-coloring of $G$ exits. The \textbf{chromatic number} of $G$, denoted $\chi(G)$, is the smallest value of $k$, such that $G$ is $k$-colorable. 

For each value of $k$, the number of $k$-colorings of $G$ will be denoted $P_G(k)$. There is a unique polynomial $P_G(t)$, which evaluated at any integer $k \geq 0$ coincides with $P_G(k)$, and is called the \textbf{chromatic polynomial} of $G$.
\end{definition}

\section{A reinterpretation of monomial invariants}

Consider a polynomial ring $S$ on $n$ variables $x_1,\ldots,x_n$. The codimension of $S/M$, denoted $\codim(S/M)$, is commonly defined as follows [Ei]. If $M$ is prime, $\codim(S/M)$ is the supremum of lengths of chains of prime ideals, descending from $M$. If $M$ is not prime, $\codim(S/M)$ is the minimum of the codimensions $\codim(S/P)$, where $P$ is a prime ideal containing $M$.

Since every prime ideal $P$ of $S$ is of the form $P=(x_{i_1},\ldots,x_{i_k})$, it follows that the supremum of lengths of chains of prime ideals descending from $P$ is $k$, the cardinality of the minimal generating set of $P$ (note that $k$ is the length of the chain 
$(0)\subseteq(x_{i_1})\subseteq (x_{i_1},x_{i_2})\subseteq \ldots\subseteq (x_{i_1},\ldots,x_{i_k}) = P$).
Thus, $\codim(S/P) = k$, the number of minimal generators of $P$.

\begin{proposition}\label{Proposition 1}
Let $M$ and $P$ be monomial ideals of $S$. Then $P$ is a prime ideal containing $M$ if and only if there is a subset $\{x_{i_1},\ldots,x_{i_k}\}$ of $\{x_1,\ldots,x_n\}$, such that $P=(x_{i_1},\ldots,x_{i_k})$, and each minimal generator of $M$ is divisible by at least one element of $\{x_{i_1},\ldots,x_{i_k}\}$.
\end{proposition}

\begin{proof}
$(\Rightarrow)$ Since every prime ideal of $S$ is minimally generated by a subset of $\{x_1,\ldots,x_n\}$, $P$ must be of the form $P=(x_{i_1},\ldots,x_{i_k})$. In addition, if $M\subseteq P$, each minimal generator of $M$ must be a multiple of a minimal generator of $P$. Therefore, each minimal generator of $M$ is divisible by at least one element of $\{x_{i_1},\ldots,x_{i_k}\}$.\\
$(\Leftarrow)$ If $P = (x_{i_1},\ldots,x_{i_k})$, then $P$ is prime. Since each minimal generator $m$ of $M$ is divisible by one of $x_{i_1},\ldots,x_{i_k}$, each $m$ must be in $P$. Hence, $M$ itself is contained in $P$.
\end{proof}

It follows from Proposition \ref{Proposition 1} that the minimum of the codimensions $\codim(S/P)$, where $P$ is prime and $P\supseteq M$, is the minimum of the cardinalities of subsets $\{x_{i_1},\ldots,x_{i_k}\}$ of $\{x_1,\ldots,x_n\}$, such that each minimal generator of $M$ is divisible by at least one of $x_{i_1},\ldots,x_{i_k}$. Now we can redefine the concept of codimension.

\begin{definition}\label{Definition codimension}
We define the \textbf{codimension} of $S/M$, denoted $\codim(S/M)$, as the minimum of the cardinalities of the subsets $X$ of $\{x_1,\ldots,x_n\}$, such that every minimal generator of $M$ is divisible by at least one element of $X$.
\end{definition}

It is convenient to point out that Proposition \ref{Proposition 1}, as well as the alternative Definition \ref{Definition codimension}, have long been known to algebraists. However, rarely do these ideas appear in the literature, and their value has perhaps gone unnoticed.

\begin{definition}
Let $M$ be squarefree. Denote $\mathscr{R}(M) =\{\{x_{i_1},\ldots,x_{i_r}\}\subseteq\{x_1,\ldots,x_n\}: r = \codim(S/M)$, and every minimal generator of $M$ is divisible by at least one of $x_{i_1},\ldots,x_{i_r}\}$.

Each set in the class $\mathscr{R}(M)$ will be called a \textbf{realization of the codimension} of $S/M$, or simply a realization of $\codim(S/M)$. 
\end{definition}

\begin{lemma}\label{Lemma 0}
Suppose that $M = (m_1,\ldots,m_s,h_1,\ldots,h_r)$ is a squarefree monomial ideal such that $r,s\geq 1$; $(h_1,\ldots,h_r)$ is a complete intersection, and $\codim(S/M) = r$. Consider the ideal 
$M_1 = (m_2,\ldots,m_s,h_1,\ldots,h_r)$, and let $M_{m_1} =(m'_2,\ldots,m'_s,h'_1,\ldots,h'_r)$, where $m'_i = \dfrac{\lcm(m_i,m_1)}{m_1}$ and $h'_i = \dfrac{\lcm(h_i,m_1)}{m_1}$. Then, 
\[\codim(S/M) = \codim(S/M_1)\leq \codim(S/M_{m_1}).\]
\end{lemma}

\begin{proof}
Since $(h_1,\ldots,h_r)\subseteq M_1\subseteq M$, we must have 
\[r = \codim\left(\dfrac{S}{(h_1,\ldots,h_r)}\right)\leq \codim\left(S/M_1\right)\leq \codim\left(S/M\right) = r.\] Thus, 
$\codim(S/M_1) = r$.

 Now, suppose that $X\subseteq \{x_1,\ldots,x_n\}$ is a realization of the codimension of $S/M_{m_1}$. Let $l\in \{m_2,\ldots,m_s,h_1,\ldots,h_r\}$, and let $l' = \dfrac{\lcm(l,m_1)}{m_1}$. Then, there is a variable $x\in X$, such that $x\mid l'$. Since $l' = \dfrac{\lcm(l,m_1)}{m_1}\mid l$, it follows that $x\mid l$. Hence, every minimal generator of $M_1$ is divisible by some element of $X$, and $r=\codim(S/M_1)\leq \left|X\right| = \codim(S/M_{m_1})$.
\end{proof}

The next theorem expresses the multiplicities of certain monomial ideals as the number of realizations of the codimensions of their polarizations. The polarization of $M$ is denoted $M_{\pol}$. 

\begin{theorem} \label{realizations}
Suppose that $M = (m_1,\ldots,m_s,h_1,\ldots,h_r)$, where $(h_1,\ldots,h_r)$ is a complete intersection, and $\codim(S/M) = r$. Then $\e(S/M) = \left|\mathscr{R}(M_{\pol})\right|$.
\end{theorem}

\begin{proof}
Let us first assume that $M$ is squarefree. The proof is by induction on $s$. 

If $s = 0$, then $M= (h_1,\ldots,h_r)$ is a complete intersection. In this case, $X$ is a realization of $\codim(S/M)$ if and only if $X$ is of the form $X=\{x_{i_1},\ldots,x_{i_r}\}$, where $x_{i_j}\mid h_j$. Since each $h_j$ is squarefree, the number of variables dividing $h_j$ equals $\deg(h_j)$. Therefore, the number of realizations of $\codim(S/M)$ is $\left|\mathscr{R}(M)\right|= \prod\limits_{j=1}^r \deg(h_j)=\e(S/M)$.

 Suppose that the theorem holds for $s-1$. We will prove it for $s$.\\
According to Lemma \ref{Lemma 0}, $\codim(S/M) = \codim(S/M_1) \leq \codim(S/M_{m_1})$. Let us consider the case $\codim(S/M) = \codim(S/M_1)<\codim(S/M_{m_1})$. By [Al1, Lemma 4.1(ii)], 
$\e(S/M) = \e(S/M_1)$. By induction Hypothesis, $\e(S/M_1) = \left|\mathscr{R}(M_1)\right|$. Thus, we need to prove that $\left|\mathscr{R}(M_1)\right| = \left|\mathscr{R}(M)\right|$. Suppose that $X$ is a realization of $\codim(S/M)$. Since $\codim(S/M_1) = \codim(S/M)$ and $M_1\subseteq M$, $X$ must be a realization of $\codim(S/M_1)$. Hence, $\mathscr{R}(M)\subseteq\mathscr{R}(M_1)$. Now, let $X$ be a realization of $\codim(S/M_1)$. Suppose, by means of contradiction, that $X\in \mathscr{R}(M_1)\setminus\mathscr{R}(M)$. Then, each of $m_2,\ldots,m_s,h_1,\ldots,h_r$ is divisible by some element of $X$, but $m_1$ is not divisible by any element of $X$. Let 
$l'\in\{m'_2,\ldots,m'_s,h'_1,\ldots,h'_r\}$. Then $l' = \dfrac{\lcm(l,m_1)}{m_1}$, where $l\in\{m_2,\ldots,m_s,h_1,\ldots,h_r\}$. Let $x\in X$ be such that $x\mid l$. Then $x\mid\lcm(l,m_1) = l' m_1$. Since $x\nmid m_1$, $x\mid l'$. We have proven that every minimal generator of $M_{m_1}$ is divisible by some element of $X$. Hence, $\codim(S/M_{m_1})\leq \left|X\right| = \codim(S/M_1)$, a contradiction. Therefore, $\mathscr{R}(M) = \mathscr{R}(M_1)$. Combining these facts, we obtain $\e(S/M) = \e(S/M_1) = \left|\mathscr{R}(M_1)\right| = \left|\mathscr{R}(M)\right|$. Now, let us consider the case $\codim(S/M) = \codim(S/M_1) = \codim(S/M_{m_1})$. By [Al1, Lemma 4.1(i)], $\e(S/M) = \e(S/M_1) - \e(S/M_{m_1})$. By induction hypothesis, $\e(S/M_1) = \left|\mathscr{R}(M_1)\right|$, and $\e(S/M_{m_1}) = \left|\mathscr{R}(M_{m_1})\right|$. Therefore $\e(S/M) = \left|\mathscr{R}(M_1)\right| - \left|\mathscr{R}(M_{m_1})\right|$. We will prove that $\left|\mathscr{R}(M)\right| = \left|\mathscr{R}(M_1)\right| - \left|\mathscr{R}(M_{m_1})\right|$. Let $X\in \mathscr{R}(M_{m_1})$. Then each of $m'_2,\ldots,m'_s,h'_1,\ldots,h'_r$ is divisible by some element of $X$, and thus, each of $m_2,\ldots,m_s,h_1,\ldots,h_r$ is divisible by some element of $X$. Since $\codim(S/M_1) = \codim(S/M_{m_1})$, it follows that $X\in \mathscr{R}(M_1)$. Therefore, $\mathscr{R}(M_{m_1})\subseteq \mathscr{R}(M_1)$. Now, let $X\in \mathscr{R}(M)$. Since $\codim(S/M) = \codim(S/M_1)$, and $M_1\subseteq M$, we must have that $X\in \mathscr{R}(M_1)$. Hence, $\mathscr{R}(M)\subseteq\mathscr{R}(M_1)$. We will show that $\mathscr{R}(M)=\mathscr{R}(M_1)\setminus\mathscr{R}(M_{m_1})$. Let $X\in \mathscr{R}(M)$, and let $x\in X$ be a variable that divides $m_1$. Since $h_1,\ldots,h_r$ is a complete intersection, there is an index $i$ such that $h_i$ is divisible by $x$, and $h_i$ is not divisible by any variable of $X\setminus\{i\}$. It follows that $h'_i$ is not divisible by any variable of $X\setminus \{x\}$. Moreover, since $x$ divides both squarefree monomials $h_i$ and $m_1$, $x\mid \gcd(h_i,m_1) = \dfrac{h_i m_1}{\lcm(h_i,m_1)} = \dfrac{h_i}{h'_i}$. Thus, $x$ must appear with exponents $1$ and $0$ in the factorizations of $h_i$ and $h'_i$, respectively. In particular, $h'_i$ is not divisible by $x$. This means that $h'_i$ is not divisible by any variable of $X$, and $X\notin\mathscr{R}(M_{m_1})$. We have proven that $\mathscr{R}(M)\subseteq\mathscr{R}(M_1)\setminus\mathscr{R}(M_{m_1})$. Now, if $X\in \mathscr{R}(M_1)\setminus\mathscr{R}(M_{m_1})$ then each of $m_2,\ldots,m_s,h_1,\ldots,h_r$ must be divisible by some element of $X$. Suppose, by means of contradiction, that $m_1$ is not divisible by any element of $X$. Let $l$ be a minimal generator of $M_1$, and let $x\in X$ be a divisor of $l$. Then, $x\mid \dfrac{\lcm(l,m_1)}{m_1} = l'$. This implies that each of $m'_2,\ldots,m'_s,h'_1,\ldots,h'_r$ is divisible by some element of $X$, which means that $X\in\mathscr{R}(M_{m_1})$, a contradiction. Thus, $m_1$ must be divisible by some element of $X$, and 
$X\in \mathscr{R}(M)$. We have proven that $\mathscr{R}(M) = \mathscr{R}(M_1)\setminus \mathscr{R}(M_{m_1})$, and given that $\mathscr{R}(M_{m_1})\subseteq \mathscr{R}(M_1)$, we have that 
$\left|\mathscr{R}(M)\right| = \left|\mathscr{R}(M_1)\right| - \left|\mathscr{R}(M_{m_1})\right|$. Combining these facts, we obtain $\e(S/M) = \left|\mathscr{R}(M_1)\right| - \left|\mathscr{R}(M_{m_1})\right| = \left|\mathscr{R}(M)\right|$.

Finally, the general case follows from the fact that $\e(S/M) =\e(S/M_{\pol})= \left|\mathscr{R}(M_{\pol})\right|$.
\end{proof}

The hypotheses of Theorem \ref{realizations} are not too restrictive. In fact, various classes of monomial ideals, including complete intersections, almost complete intersections, and artinian ideals, satisfy these conditions. (Moreover, we conjecture that Theorem \ref{realizations} holds for arbitrary monomial ideals.) 

\section{Chromatic ideals}

In this pivotal section, we define chromatic ideals, the algebraic counterpart of graphs.

\begin{definition}\label{DefChrom}
Given a graph $G$ on $n$ vertices, let $\Omega = \{\omega\subseteq V: \omega$ is a maximal independent set$\}$, and $\Gamma = \left\{\{1\},\{2\},\ldots,\{n\}\right\}\cup \Omega.$
We define the \textbf{chromatic ideal} of $G$, denoted $M_G$, as the monomial ideal $M_G = (m_1,\ldots,m_n)$, where $m_i = \prod\limits_{\substack{\omega\in\Gamma\\ i\in\omega}} x_{\omega}$. (We will say that the vertex $i$ and the monomial $m_i$ are \textbf{associated}.)
\end{definition}

\begin{example} \label{Example graph}
Let $G$ be the following graph:

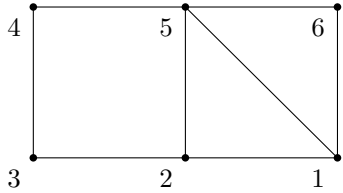
\[\begin{tikzpicture}
\draw (0,0) rectangle (4,2);
\draw (2,0)-- (2,2);
\draw (2,2) -- (4,0);
 \node[circle, label=below left:{1}, fill=black, inner sep=1pt] at (4,0) {};
 \node[circle, label=below left:{2}, fill=black, inner sep=1pt] at (2,0) {};
  \node[circle, label=below left:{3}, fill=black, inner sep=1pt] at (0,0) {};
   \node[circle, label=below left:{4}, fill=black, inner sep=1pt] at (0,2) {};
    \node[circle, label=below left:{5}, fill=black, inner sep=1pt] at (2,2) {};
     \node[circle, label=below left:{6}, fill=black, inner sep=1pt] at (4,2) {};
\end{tikzpicture}\]
Let us construct the chromatic ideal of $G$. Notice that 
\[\Omega = \left\{\{1,3\},\{1,4\},\{2,4,6\},\{3,5\},\{3,6\}\right\}.\] Hence, $M_G = (m_1,\ldots,m_6)$, where 
\[
\begin{array}{ll}
m_1 = x_{\{1\}} x_{\{1,3\}} x_{\{1,4\}}& m_4 = x_{\{4\}}x_{\{1,4\}} x_{\{2,4,6\}}\\
m_2 = x_{\{2\}} x_{\{2,4,6\}}  & m_5 = x_{\{5\}}x_{\{3,5\}}\\
m_3 = x_{\{3\}}x_{\{1,3\}} x_{\{3,5\}} x_{\{3,6\}} & m_6 = x_{\{6\}} x_{\{3,6\}} x_{\{2,4,6\}}.
\end{array}\]
By doing the change of variables $a = x_{\{1,3\}}$, $b = x_{\{1,4\}}$, $c = x_{\{2,4,6\}}$, $d=x_{\{3,5\}}$, $e = x_{\{3,6\}}$, $f = x_{\{1\}}$, $g = x_{\{2\}}$, $h = x_{\{3\}}$, $i = x_{\{4\}}$, 
$j = x_{\{5\}}$, and $k = x_{\{6\}}$, $M_G$ can be expressed in the form
$M_G = (fab,gc,hade,ibc,jd,kce)$.

\end{example}

\begin{proposition}\label{Prop}
Suppose that the chromatic ideal $M_G$ of a graph $G$ is minimally generated by $n$ monomials $m_1,\ldots,m_n$. Then, $G = (V,E)$, where $V = \{1,\ldots,n\}$ and 
$E = \left\{\{i,j\}\subseteq V: \lcm(m_i,m_j) = 1\right\}$.
\end{proposition}

\begin{proof}
Let $V = \{1,\ldots,k\}$. By construction, $M_G$ is generated by the $k$ monomials associated to $1,2,\ldots,k$. Denote these monomials by $h_1,\ldots,h_k$. Then, $n\leq k$. Suppose that $n<k$. This means that one of the generators $h_1,\ldots,h_k$ is not a minimal generator. That is, there are generators $h_i$ and $h_j$, with $i\ne j$, such that $h_i\mid h_j$. By construction, if $1\leq r$, $s\leq k$, then $x_{\{r\}}\mid h_s$ if and only if $r = s$. Therefore, $x_{\{i\}}\mid h_i$, but $x_{\{i\}}\nmid h_j$, which means that $h_i\nmid h_j$, a contradiction. We conclude that $n = k$, and $V = \{1,\ldots,n\}$.

Now, if $\{i,j\} \in E$, then no independent set $\omega$ contains $i$ and $j$, simultaneously. Therefore, no variable $x_{\omega}$ can be a common divisor of $m_i$ and $m_j$. Thus, $\lcm(m_i,m_j) = 1$. Conversely, if $\lcm(m_i,m_j) = 1$, we must have that no variable $x_{\omega}$ divides both $m_i$ and $m_j$. In turn, this implies that no maximal independent set $\omega$ contains both $i,j$. It follows that $i$ and $j$ cannot belong to the same independent set. In particular, $\{i,j\}$ is not an independent set, and thus, $\{i,j\}\in E$.
\end{proof}

In light of Proposition \ref{Prop}, we conclude that a graph $G$ and its chromatic ideal $M_G$ yield the same information. In fact, $M_G$ can be obtained from $G$, and vice versa. 

\begin{definition}
Let $M$ be minimally generated by monomials $m_1, \ldots, m_n$. A minimal generator $m_i$ is called \textbf{dominant}, if there is a variable $x$ such that the exponent with which $x$ appears in the factorization of $m_i$ is larger than the exponent with which $x$ appears in the factorization of $m_j$, for all $j \neq i$. If each minimal generator of $M$ is dominant, then we say that $M$ is a dominant ideal.
\end{definition}

\begin{example}
Let $M=(m_1=a^2,m_2=ab^2,m_3=bc)$. Notice that $m_1$ is dominant, for the exponent with which $a$ appears in the factorization of $m_1$ is larger than the exponent with which $a$ appears in the factorizations of $m_2$ and $m_3$. With a similar argument, it can be shown that $m_2$ and $m_3$ are also dominant generators and therefore, $M$ is a dominant ideal.
\end{example}

An interesting feature of dominant ideals is that they characterize when the Taylor resolution is minimal [Al]. The next result shows that chromatic ideals are dominant.

\begin{proposition}\label{chrDOM}
For every graph $G$, the chromatic ideal $M_G$ is dominant.
\end{proposition}
\begin{proof}
Let $1,\ldots,n$ be the vertices of $G$, and denote by $m_1,\ldots,m_n$ their associated monomials. Then, $M=(m_1,\ldots,m_n)$. Choose a generator $m_i$. By construction, the variable $x_{\{i\}}$ appears in the factorization of $m_i$, but not in the factorization of $m_j$, for all $j \neq i$. Therefore, $m_i$ is both a minimal generator and a dominant generator. Since $i$ is arbitrary, $M_G$ is a dominant ideal.
\end{proof}

Next, we give another simple but important property of chromatic ideals.

\begin{proposition}\label{PropCI}
If $G$ is a graph containing a $k$-clique as a subgraph, then its chromatic ideal $M_G$ can be expressed in the form $M_G = (m_1,\ldots,m_s,h_1,\ldots,h_k)$, where $(h_1,\ldots,h_k)$ is a complete intersection.
\end{proposition}

\begin{proof}
Let $1, \ldots, k$ be the vertices of a $k$-clique contained in $G$, and denote by $h_1,\ldots,h_k$ their associated monomials. Consider two of these monomials, say $h_i$ and $h_j$, and a variable $x_\omega$ dividing $h_i$. Notice that $i \in \omega$. Since $i$ and $j$ are vertices of the same clique, they must be adjacent and hence, no independent set contains both $i$ and $j$. It follows that $j \notin \omega$ and thus, $x_{\omega} \nmid h_j$. This implies that $\lcm(h_i,h_j) = 1$. Since $i$ and $j$ are arbitrary, $(h_1,\ldots,h_k)$ must be a complete intersection.
\end{proof}

The following theorem proves that the chromatic number of a graph is equal to the codimension of its chromatic ideal, which explains the term chromatic ideal.

\begin{theorem}\label{Theorem 3}
For an arbitrary graph $G$, $\chi(G) = \codim(S/M_G)$.
\end{theorem}

\begin{proof}
Let $r = \chi(G)$. Then, there is a partition $V = W_1\cup\ldots\cup W_r$ of the vertex set of $G$, into $r$ independent sets. For each $1\leq i\leq r$, let  $\omega_i$ be a maximal independent set containing $W_i$. Let $m_s$ be a minimal generator of $M$, and let $s$ be the vertex associated to $m_s$. Then, there is an $1\leq i\leq r$, such that $s\in W_i \subseteq \omega_i$. It follows that the variable $x_{\omega_i}$ appears in the factorization of $m_s$, which proves that every minimal generator of $M_G$ is divisible by one of $x_{\omega_1}, \ldots, x_{\omega_r}$. Hence, $\codim(S/M_G) \leq r$. 

Suppose that $\codim(S/M_G) = h< r$. Then, there exists a realization $\{x_{\gamma_1}, \ldots, x_{\gamma_h}\}$ of  $\codim(S/M)$. In particular, $\gamma_1, \ldots, \gamma_h$ are independent sets. Consider a vertex $s$ and its associated monomial $m_s$. Then, for some $1\leq i\leq h$, $m_s$ is divisible by $x_{\gamma_i}$, which implies that $s\in \gamma_i$. This proves that $V = \bigcup\limits_{i=1}^h {\gamma}_i$. Let ${\gamma}'_1={\gamma}_1$, ${\gamma}'_2={\gamma}_2 \setminus {\gamma}_1$, and more generally, ${\gamma}'_i = {\gamma}_i\setminus \bigcup\limits_{j=1}^{i-1} {\gamma}_j$. Then, $\{ {\gamma}'_1, \ldots, {\gamma}'_h \}$ is a partition of $V$ into $h$ independent sets, which implies that $\chi(G)\leq h<r$, a contradiction. This shows that $\chi(G) =r= \codim(S/M_G)$.
\end{proof}

Given the strong connection between a graph $G$ and its chromatic ideal $M_G$, it is not difficult to interpret invariants of $G$ in terms of invariants of $M_G$. One instance of this interplay is seen in Theorem \ref{Theorem 3}. Here is another example. The \textbf{degree} of a vertex $i$ of $G$, denoted $\deg(i)$, is the number of edges that are incident to $i$. It follows from Proposition \ref{Prop} that $\deg(i)$ can be expressed algebraically as the number of minimal generators $m_j$ of $M_G$ such that $\lcm(m_i,m_j) = 1$.

The advantage of interpreting graph invariants in terms of monomial invariants is the fact that we can put the theory of monomial ideals to the service of graph theory (and vice versa), which we do in the next section.

\section{Multiplicities and the chromatic polynomial}

The interaction between a graph and its chromatic ideal reaches its climax in the next theorem, where the chromatic number and the chromatic polynomial of $G$, as well as the codimension and multiplicity of $M_G$ are intertwined to create a useful formula.

\begin{theorem}\label{Theorem 5.6}
Suppose that a graph $G$ can be expressed as the union of finitely many $k$-cliques. If $\chi(G) = k$, then $P_G(k) = k! \e(S/M_G)$.
\end{theorem}

\begin{proof}
Denote by $\{K(1),\ldots,K(r)\}$ the finite class of $k$-cliques stated by the theorem. Since $P_G(k)$ equals the number of configurations of $k$-colorings of $G$ multiplied by the number of permutations of $k$ colors, it suffices to prove that $\e(S/M_G)$ equals the number of configurations of $k$-colorings of $G$. That is, we only need to prove that $\e(S/M_G)$ equals the number of partitions of $V$ into $k$ independent sets. Since $G$ contains a $k$-clique as subgraph, it follows from Proposition  \ref{PropCI} that the minimal generating set of $M_G$ must contain a regular sequence of $k$ elements. By Theorem \ref{Theorem 3}, $\codim(S/M_G) = \chi(G) = k$, and by Theorem \ref{realizations}, $\e(S/M_G)$ equals the number of realizations of $\codim(S/M_G)$. Therefore, the theorem will be proven if we show that the rule 
\[\{W_1,\ldots,W_k\}\xrightarrow{f} \{x_{W_1},\ldots,x_{W_k}\}\]
defines a bijective correspondence between the class of partitions of $V$ into $k$ independent sets, and the class of realizations of $\codim(S/M_G)$.\\
The proof consists of three steps. First, $f$ is well-defined. Let $V = W_1\cup\ldots\cup W_k$ be a partition of $V$ into $k$ independent sets. We will prove that $W_1,\ldots,W_k$ are maximal independent sets. Suppose not. Then for some $i$, $W_i$ is strictly contained in an independent set $W$. Let $v\in W\setminus W_i$, and let $K(s)$ be a $k$-clique containing $v$. Since $W$ is independent, no vertex of $W\setminus\{v\}$ is in $K(s)$. In particular, $W_i\cap K(s) = \varnothing$. It follows that 
$K(s) = \left(W_1\cap K(s)\right) \cup \ldots \cup\widehat{ \left(W_i\cap K(s)\right)}\cup\ldots \cup \left(W_k\cap K(s)\right)$ is a partition of $K(s)$ into $k-1$ independent sets, a contradiction. Thus, $W_1,\ldots,W_k$ must be maximal independent sets. 
For each vertex $v\in V$, let $W$ be a maximal independent set containing $v$. Then, the corresponding monomial generator $m_v$ of $M_G$ is divisible by $x_W$. Hence, $\{x_{W_1},\ldots,x_{W_k}\}$ is a realization of $\codim(S/M_G)$.\\
Next, $f$ is one-to-one. If $\{x_{Y_1},\ldots,x_{Y_k}\} = \{x_{Z_1},\ldots,x_{Z_k}\}$, then $\{Y_1,\ldots,Y_k\} = \{Z_1,\ldots,Z_k\}$.\\
Finally, $f$ is onto. Let $\{x_{W_1},\ldots,x_{W_k}\}$ be a realization of $\codim(S/M_G)$. Then, $W_1,\ldots,W_k$ are independent sets, and $V = \bigcup\limits_{i=1}^k W_i$. Moreover, $V= W_1\cup (W_2\setminus W_1)\cup\ldots \cup\left (W_k\setminus \bigcup\limits_{i=1}^{k-1} W_i\right)$ is a partition of $V$ into $k$ independent sets. Suppose that, for some $1\leq s\leq k$, there is a vertex $v\in V$, such that $v\in W_s \cap\left(\bigcup\limits_{i=1}^{s-1} W_i\right)$. Let $K(j)$ be a $k$-clique such that $v\in K(j)$. Since $W_s$ is independent, $\left(W_s\setminus \bigcup\limits_{i=1}^{s-1}W_i\right)\cap K(j)= \varnothing$. Therefore, 
	\[K(j) =\left [W_1\cap K(j)\right]\cup \ldots \cup \left [\left(W_s\setminus \bigcup\limits_{i=1}^{s-1}W_i\right)\cap K(j)\right]^{\wedge}\cup \ldots \cup \left[\left(W_k\setminus \bigcup\limits_{i=1}^{k-1}W_i\right)\cap K(j)\right]\]
is a partition of $K(j)$ into $k-1$ independent sets, a contradiction. It follows that $W_s\cap \left(\bigcup\limits_{i=1}^{s-1} W_i\right) = \varnothing$, which implies that the sets $W_1,\ldots,W_k$ are pairwise disjoint. Then, $V = W_1\cup (W_2\setminus W_1)\cup\ldots \cup \left(W_k\setminus \bigcup\limits_{i=1}^{k-1} W_i\right) = W_1\cup\ldots\cup W_k$ is a partition of $V$ into $k$ independent sets.
\end{proof}

The next theorem, stated below for completeness, can be found in [Al1, Corollary 5.4].

\begin{theorem}\label{MonomialMult}
  Suppose that $M$ is a dominant ideal that can be expressed in the form $M=(m_1,\ldots,m_s,h_1,\ldots,h_k)$, where $k=\codim(S/M)$, and $(h_1,\ldots,h_k)$ is a complete intersection. Then,
  \[\e(S/M) = \sum_{\substack{1\le r_1<\cdots<r_j\le s\\0\le j\le s}} (-1)^j \prod\limits_{i=1}^k \deg\left(\dfrac{\lcm(m_{r_1},\ldots,m_{r_j},h_i)}{\lcm(m_{r_1},\ldots,m_{r_j})}\right).\]
\end{theorem}

\begin{corollary}\label{C}
Suppose that a graph $G$ can be expressed as the union of finitely many $k$-cliques. If $\chi(G) = k$, then
\begin{enumerate}[(i)]
\item $M_G$ can be expressed in the form $M_G = (m_1,\ldots,m_s,h_1,\ldots,h_k)$, where $(h_1,\ldots,h_k)$ is a complete intersection.
\item $P_G(k) = k! \sum\limits_{\substack{1\le r_1<\cdots <r_j\le s\\ 0\le j\le s}} (-1)^j \prod\limits_{i=1}^k \deg\left(\dfrac{\lcm(m_{r_1},\ldots,m_{r_j},h_i)}{\lcm(m_{r_1},\ldots,m_{r_j})}\right)$.
\end{enumerate}
\end{corollary}

\begin{proof}
(i) It follows immediately from Proposition \ref{PropCI}. \\
(ii) By Proposition \ref{chrDOM}, $M_G$ is dominant and, by Theorem \ref{Theorem 3}, $\codim(S/M_G) = k$. Therefore, by Theorem \ref{MonomialMult}, we have
\[\e(S/M_G) = \sum\limits_{\substack{1\le r_1<\cdots <r_j\le s\\ 0\le j\le s}} (-1)^j \prod\limits_{i=1}^k \deg\left(\dfrac{\lcm(m_{r_1},\ldots,m_{r_j},h_i)}{\lcm(m_{r_1},\ldots,m_{r_j})}\right). \]
Now, our claim follows directly from Theorem \ref{Theorem 5.6}.
\end{proof}

\section{The Erdős-Faber-Lovász conjecture}

We will explore classes of ideals that satisfy the hypotheses of Theorem \ref{Theorem 5.6} and Corollary \ref{C}. Our main source of examples is derived from the following conjecture, originally posed by P. Erdős, V. Faber, and L. Lovász, in 1972 [Er]. 

\begin{conjecture}\label{conjecture}
If $G$ is the union of $k$ $k$-cliques that intersect pairwise in at most one vertex, then its chromatic number is $\chi(G)=k$.
\end{conjecture}

One particular case where the Erdős-Faber-Lovász conjecture is known to hold is due to N. Hindman, and states that the conjecture is true if no vertex belongs to more than two of the $k$-cliques [Hi]. In order to have concrete instances where the formula of Corollary \ref{C} can be verified, we will consider this case and will use it to prove a related result. The next theorem is discussed in [Hi, Section 1].

\begin{theorem} \label{hindman}
Suppose that $G$ is a graph obtained as the union of $k$ $k$-cliques, having the following properties:
\begin{enumerate}[(i)]
\item the $k$-cliques intersect pairwise in at most one vertex,
\item every vertex of $G$ is shared by at most 2 of the $k$-cliques.
\end{enumerate}
Then, $\chi(G)=k$.
\end{theorem}

We will show that when $k$ is odd a result slightly stronger than Theorem \ref{hindman} holds. To prove this fact, we need to introduce new notation, and to prove two lemmas.\\

\textit{Notation}: For each integer $r$, let $\mathscr{F}_r = \{\{i,j\}: 1\leq i<j\leq r\}$.

\begin{lemma}\label{Lemma 1}
Let $k$ be an integer. Let $f:\mathscr{F}_k\rightarrow \{0,\ldots,k-1\}$ be defined by $f(\{i,j\}) = u$, where $i+j\equiv u \mod k$. Then, for each $0\leq u\leq k-1$, the elements of $f^{-1}(u)$ are disjoint.
\end{lemma}

\begin{proof}
Suppose that $f(\{i,j\}) = f(\{s,t\}) = u$. If $\{i,j\}$ and $\{s,t\}$ have an element in common, then they must be of the form $\{\alpha,\beta\}$ and $\{\alpha,\gamma\}$, respectively. Then 
$k\mid (\alpha+\beta) - (\alpha+\gamma) = \beta - \gamma$. It follows that $\beta = \gamma$, and $\{i,j\} = \{\alpha,\beta\} = \{s,t\}$. 
\end{proof}

\begin{lemma}\label{Lemma 2}
Let $k$ be odd. Let $h:\mathscr{F}_{k+1}\rightarrow \{0,\ldots,k-1\}$ be defined by
\[h(\{s,t\}) = 
\begin{cases}
(s+t)\mod k &\text{ if } \{s,t\}\in\mathscr{F}_k\\
2s    & \text{ if } 1\leq s\leq \dfrac{k-1}{2} \text{ and } t = k+1\\
2s-k & \text{ if } \dfrac{k-1}{2}< s\leq k-1 \text{ and } t=k+1\\
0 & \text{ if } s=k \text{ and } t=k+1.
\end{cases}\]
Then, for each $0\leq u\leq k$, the elements of $h^{-1}(u)$ are pairwise disjoint.
\end{lemma}

\begin{proof}
 Let $f:\mathscr{F}_k\rightarrow \{0,\ldots,k-1\}$ be defined by $f(\{s,t\}) = u$, where $s+t\equiv u \mod k$.\\
Notice that $\{0,\ldots,k-1\}$ is the disjoint union of the sets $A=\{u: u\text{ is even and }2\leq u\leq k-1\}$, $B=\{u: u\text{ is odd and }1\leq u\leq k-2\}$, and $C=\{0\}$. Thus, the lemma can be proven by considering the following 3 cases.

\underline{Case 1}:  for each $u\in A$, the elements of $h^{-1}(u)$ are disjoint.\\
Since $A$ can be represented in the form $A=\left\{2s:1\leq s\leq \dfrac{k-1}{2}\right\}$, it is enough to show that for each $1\leq s\leq \dfrac{k-1}{2}$, the elements of $h^{-1}(2s)$ are disjoint. Note that $h^{-1}(2s)=f^{-1}(2s)\cup \{\{s,k+1\}\}$ and, by Lemma \ref{Lemma 1}, the elements of $f^{-1}(2s)$ are disjoint. Hence, this case reduces to showing that $\{s,k+1\} \cap \{i,j\}=\varnothing$, for all $\{i,j\} \in f^{-1}(2s)$. Suppose not. Then $f^{-1}(2s)$ must contain an element of the form $\{s,\gamma\}$. Since $f(\{s,\gamma\}) = 2s$, we must have that $k\mid (s+\gamma)-2s = \gamma-s$. It follows that $\gamma = s$, and $\{s,\gamma\} = \{s,s\}$, a contradiction.\\ 

\underline{Case 2}:  for each $u\in B$, the elements of $h^{-1}(u)$ are disjoint.\\
Since $B$ can be represented in the form $B=\{2s-k:\dfrac{k-1}{2}< s\leq k-1\}$, it is enough to show that for each $\dfrac{k-1}{2}< s\leq k-1$, the elements of $h^{-1}(2s-k)$ are disjoint. Note that $h^{-1}(2s-k)=f^{-1}(2s-k)\cup \{\{s,k+1\}\}$ and, by Lemma \ref{Lemma 1}, the elements of $f^{-1}(2s-k)$ are disjoint. Hence, this case reduces to showing that $\{s,k+1\} \cap \{i,j\}=\varnothing$, for all $\{i,j\} \in f^{-1}(2s-k)$. Suppose not. Then $f^{-1}(2s-k)$ must contain an element of the form $\{s,\gamma\}$. Since $f(\{s,\gamma\}) = 2s-k$, we must have that $k\mid (s+\gamma)-(2s-k) = \gamma-s+k$. It follows that $\gamma = s$, and $\{s,\gamma\} = \{s,s\}$, a contradiction.\\ 

\underline{Case 3}:  the elements of $h^{-1}(0)$ are disjoint.\\
Note that $h^{-1}(0)=f^{-1}(0)\cup \{\{k,k+1\}\}$ and, by Lemma \ref{Lemma 1}, the elements of $f^{-1}(0)$ are disjoint. Hence, this case reduces to showing that $\{k,k+1\} \cap \{i,j\}=\varnothing$, for all $\{i,j\} \in f^{-1}(0)$. Suppose not. Then $f^{-1}(0)$ must contain an element of the form $\{k,\gamma\}$. Since $f(\{k,\gamma\}) = 0$, we must have that $k\mid (k+\gamma)$. It follows that $\gamma = k$, and $\{k,\gamma\} = \{k,k\}$, a contradiction.\\ 

\end{proof}

\begin{example}
Let $k=7$. The functions $f$ and $h$, defined by Lemmas \ref{Lemma 1} and \ref{Lemma 2}, are
\[
\begin{array}{cccccccc}
     &   &       &    &   &    &    &  \\
     &\vline\;34\;\vline&\vline\;35\;\vline&\vline\;45\;\vline&\vline\;46\;\vline&\vline\;56\;\vline&\vline\;57\;\vline&\vline\;67\;\vline\\
     &\vline\;25\;\vline&\vline\;26\;\vline&\vline\;36\;\vline&\vline\;37\;\vline&\vline\;47\;\vline&\vline\;23\;\vline&\vline\;24\;\vline\\
     &\vline\;16\;\vline&\vline\;17\;\vline&\vline\;27\;\vline&\vline\;12\;\vline&\vline\;13\;\vline&\vline\;14\;\vline&\vline\;15\;\vline\\
f    &    \downarrow&\downarrow&\downarrow&\downarrow&\downarrow&\downarrow&\downarrow\\
  &  0               &1               &2                  &3              &  4               &   5            &   6            
\end{array}
\begin{array}{cccccccc}
      &\vline\;78\;\vline   &\vline\;48\;\vline    &\vline\;18\;\vline    &\vline\;58\; \vline &\vline\;28 \;\vline &\vline  \;68 \;\vline &\vline \;38\;\vline\\
     &\vline\;34\;\vline&\vline\;35\;\vline&\vline\;45\;\vline&\vline\;46\;\vline&\vline\;56\;\vline&\vline\;57\;\vline&\vline\;67\;\vline\\
     &\vline\;25\;\vline&\vline\;26\;\vline&\vline\;36\;\vline&\vline\;37\;\vline&\vline\;47\;\vline&\vline\;23\;\vline&\vline\;24\;\vline\\
     &\vline\;16\;\vline&\vline\;17\;\vline&\vline\;27\;\vline&\vline\;12\;\vline&\vline\;13\;\vline&\vline\;14\;\vline&\vline\;15\;\vline\\
h    &    \downarrow&\downarrow&\downarrow&\downarrow&\downarrow&\downarrow&\downarrow\\
  &  0               &1               &2                  &3              &  4               &   5            &   6             
\end{array}
\]

\end{example} 

\textit{Note}: Lemma \ref{Lemma 2} is false when $k$ is even. In fact, when $k=2$, $\mathscr{F}_3 = \{\{1,2\},\{1,3\},\{2,3\}\}$ and it is impossible to define a function $h$ as in Lemma \ref{Lemma 2}.\\

The next theorem, which is a strengthened version of Theorem \ref{hindman} when $k$ is odd, is proven with the same line of reasoning that Hindman used in [Hi, Section 1].

\begin{theorem}\label{Theorem 4}
Let $k$ be odd. Suppose that $G$ is a graph obtained as the union of $k+1$ $k$-cliques, having the following properties:
\begin{enumerate}[(i)]
\item the $k$-cliques intersect pairwise in at most one vertex,
\item every vertex of $G$ is shared by at most $2$ of the $k$-cliques.
\end{enumerate}
Then $\chi(G) = k$.
\end{theorem} 

\begin{proof}
Denote by $K(1),\ldots,K(k+1)$ the $k+1$ $k$-cliques stated by the theorem. Let $G=(V,E)$. Consider the set $V'\subseteq V$ of vertices that are shared by exactly 2 of the $K(i)$, and let $G'$ be the induced subgraph of $G$ on the vertex set $V'$. Suppose that there is a $k$-coloring $f:V'\rightarrow \{0,\ldots,k-1\}$ of $G'$. Then $f$ can be extended to a $k$-coloring of $G$ as follows. For each $i=1,\ldots,k+1$, define a bijection $g_i:K(i)\rightarrow \{0,\ldots,k-1\}$ such that $g_i(v)=f(v)$, for all $v \in V' \cap K(i)$. (Since $f$ is a $k$-coloring of $V'$, $f$ is one-to-one in $V' \cap K(i)$. Thus, the bijection $g_i$ actually exists.) Now, let $g:V\rightarrow \{0,\ldots,k-1\}$ be the (only) function whose restriction to $K(i)$ equals $g_i$. Let $\{v_1,v_2\} \in E$. Then, $v_1, v_2 \in K(i)$, for some $i$. It follows that $g(v_1)=g_i(v_1)\neq g_i(v_2)=g(v_2)$, which shows that $g$ is a $k$-coloring of $G$.

Thus, the theorem will be proven if we can construct a $k$-coloring $f:V'\rightarrow \{0,\ldots,k-1\}$ of $G'$. Let $h$ be the function defined in Lemma \ref{Lemma 2}, and consider the function $h':V'\rightarrow \mathscr{F}_{k+1}$, defined by the following rule: if $v \in K(i) \cap K(j)$, then $h'(v)=\{i,j\}$. Let $f=h \circ h'$, and suppose that $f(v_1)=f(v_2)$, for distinct vertices $v_1$, $v_2$ of $V'$. Denote $h'(v_1)=\{i,j\}$, and $h'(v_2)=\{r,s\}$. Then $h(\{i,j\})= h(\{r,s\})$. By Lemma \ref{Lemma 2}, $\{i,j\} \cap \{r,s\}= \varnothing$, which means that $v_1 \in K(i) \cap K(j)$,  $v_2 \in K(r) \cap K(s)$, and $\{K(i),K(j)\} \cap \{K(r),K(s)\}= \varnothing$. By property (ii), $v_1$ and $v_2$ cannot be vertices of the same $k$-clique, which implies that $v_1$ and $v_2$ are not adjacent. This proves that $f$ is a $k$-coloring of $G'$.
\end{proof}

\textit{Note}: Theorem \ref{Theorem 4} does not hold when $k$ is even, the simplest counterexample being $k=2$, and $G = \left(V=\{x_1,x_2,x_3\},E=\left\{\{x_1,x_2\},\{x_1,x_3\},\{x_2,x_3\}\right\}\right)$.\\

The next example illustrates Theorem \ref{Theorem 4} by giving a $k$-coloring of a graph $G$, explicitly.

\begin{example}
Let $G$ be the union of the following $4$ $3$-cliques:
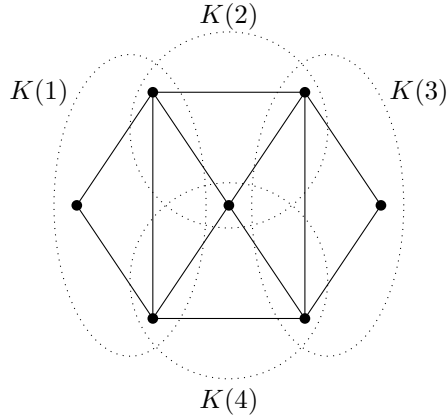
\[
\begin{tikzpicture}
\draw (0,0) --(2,0)-- (2,3)--(0,3)--(0,0);
        \fill (0,0) circle [radius=2pt];
     \fill (2,0)  circle [radius=2pt];
      \fill (2,3)  circle [radius=2pt];
       \fill (0,3)  circle [radius=2pt];
\draw (2,0) -- (3,1.5);
\fill (3,1.5)  circle [radius=2pt];
\draw (3,1.5) -- (2,3);
\draw (0,0)--(2,3);
\draw (2,0)--(0,3);
\fill (intersection of 0,0--2,3 and 2,0--0,3) circle (2pt);
\draw [ dotted ](2.3,1.5) ellipse (1 cm and 2cm);
\node at (3.5,3){$K(3)$};
\node at (1,4){$K(2)$};
\node at (-1.5,3) {$K(1)$};
\node at (1,-1.1){$K(4)$};
\draw [ dotted ](1,2.5) circle (1.3);
\draw (0,0)--(-1,1.5)--(0,3);
\fill (-1,1.5) circle [radius=2pt];
\draw [ dotted ](-0.3,1.5) ellipse (1 cm and 2cm);
\draw [ dotted ](1,0.5) circle (1.3 cm );
\end{tikzpicture}\]
With the notation of Theorem \ref{Theorem 4}, $G'$ is the graph
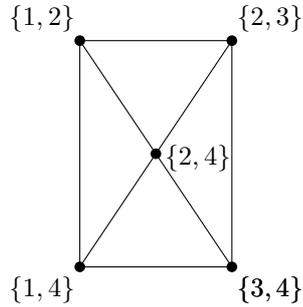
\[
\begin{tikzpicture}
\draw (0,0) rectangle (2,3);
 \fill (0,0) circle [radius=2pt];
     \fill (2,0)  circle [radius=2pt];
      \fill (2,3)  circle [radius=2pt];
       \fill (0,3)  circle [radius=2pt];
       \draw (0,0)--(2,3);
       \draw (2,0)--(0,3);
       \fill (intersection of 0,0--2,3 and 2,0--0,3) circle (2pt);
       \node at (1.55,1.45) {$\{2,4\}$};
       \node at (-0.5,-0.3) {$\{1,4\}$};
       \node at (-0.5,3.3) {$\{1,2\}$};
       \node at (2.5,-0.3) {$\{3,4\}$};
       \node at (2.5,3.3) {$\{2,3\}$};
              \node at (2.5,-0.3) {$\{3,4\}$};
\end{tikzpicture}
\]
(Here we depart from the standard notation, and denote vertices with $2$-element sets.) By Lemma \ref{Lemma 2}, the following rule defines a $k$-coloring of $G'$.
\[
\begin{array}{cccccc}
  &\{3,4\}&\vline&\{2,4\}&\vline&\{1,4\}\\
  &\{1,2\}&\vline&\{1,3\}&\vline&\{2,3\}\\
h&\downarrow&&\downarrow&&\downarrow\\
  &0&&1&&2
\end{array}\]
Finally, each clique $K(i)$ of $G$ can be $3$-colored as follows. Color the vertices of $K(i)$ that are in $G'$, as the function $h$ indicates. If $A_i$ is the set of all integers used to color these vertices, then color the remaining vertices of $K(i)$ with the integers of $\{0,1,2\}\setminus A_i$, arbitrarily. Thus, $G$ can be $3$-colored in the following fashion. 
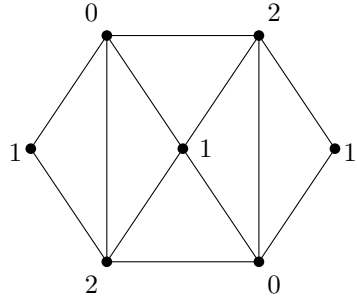
\[
\begin{tikzpicture}
\draw (0,0) rectangle (2,3);
 \fill (0,0) circle [radius=2pt];
     \fill (2,0)  circle [radius=2pt];
      \fill (2,3)  circle [radius=2pt];
       \fill (0,3)  circle [radius=2pt];
       \draw (0,0)--(2,3);
       \draw (2,0)--(0,3);
       \fill (intersection of 0,0--2,3 and 2,0--0,3) circle (2pt);
       \node at (1.3,1.5) {$1$};
       \node at (-0.2,-0.3) {$2$};
       \node at (-0.2,3.3) {$0$};
       \node at (2.2,-0.3) {$0$};
       \node at (2.2,3.3) {$2$};
       \draw (2,0) -- (3,1.5);
\fill (3,1.5)  circle [radius=2pt];
\draw (3,1.5) -- (2,3);
\draw (0,0)--(-1,1.5)--(0,3);
\fill (-1,1.5) circle [radius=2pt];
\node at (-1.2,1.45) {$1$};
\node at (3.2,1.45) {$1$};
\end{tikzpicture}
\]

\end{example}

We close this section applying the formula of Theorem \ref{Theorem 5.6} to two graphs.

\begin{example}
(i) Consider the $3$-cliques $K(1)=\left(\{1,2,3\}, \{\{1,2\},\{1,3\},\{2,3\}\}\right)$, $K(2)=\left(\{3,4,5\}, \{\{3,4\},\{3,5\},\{4,5\}\}\right)$, and $K(3)=\left(\{1,5,6\}, \{\{1,5\},\{1,6\},\{5,6\}\}\right)$. Let $G=K(1)\cup K(2)\cup K(3)$. 
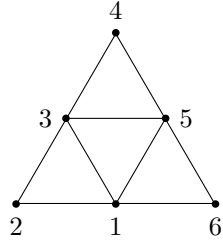
\[\begin{array}{c}
\begin{tikzpicture}
\foreach \a in {3}{
\node[regular polygon, regular polygon sides=\a, minimum size=3cm, draw]  (A) {};
    \foreach \i in {1}
    \node[circle, label=above:{4}, fill=black, inner sep=1pt] at (A.corner \i) {};
    \foreach \i in {2}
    \node[circle, label=below:{2}, fill=black, inner sep=1pt] at (A.corner \i) {};
     \foreach \i in {3}
    \node[circle, label=below:{6}, fill=black, inner sep=1pt] at (A.corner \i) {};
    \draw (A.side 1)--(A.side 3);
    \node[circle, label=left:{3}, fill=black, inner sep=1pt] at (A.side 1) {};
    \draw (A.side 1)--(A.side 2);
    \node[circle, label=below:{1}, fill=black, inner sep=1pt] at (A.side 2) {};
    \draw (A.side 2)--(A.side 3);
    \node[circle, label=right:{5}, fill=black, inner sep=1pt] at (A.side 3) {};
   }
\end{tikzpicture}\\
G=K(1)\cup K(2)\cup K(3)
\end{array}\]

 Notice that any 2 of $K(1)$, $K(2)$, $K(3)$ intersect in exactly one vertex. Thus, the graph $G = K(1)\cup K(2)\cup K(3)$ satisfies the hypotheses of the Erdős-Faber-Lovász conjecture. Moreover, each vertex of $G$ belongs to no more than $2$ of $K(1)$, $K(2)$, $K(3)$. By Theorem \ref{hindman}, $\chi(G)=3$.

We will compute $P_3(G)$ in two different ways. First, note that the function $f:\{1,\ldots,6\}\rightarrow \{0,1,2\}$,defined by $f(1) = f(4)=0$; $f(2) =f(5)=1$; $f(3)=f(6)=2$ is a $3$-coloring. Also, note that in any $3$-coloring $g:\{1,\ldots,6\}\rightarrow \{0,1,2\}$, $g(1)\ne g(3)$, $g(3)\ne g(5)$, and $g(1)\ne g(5)$ because the vertices $1,3,5$ are pairwise adjacent. It follows that 
$\{g(1),g(3),g(5)\}=\{0,1,2\}$, and hence, $g(2),g(4),g(6)\in \{g(1),g(3),g(5)\}$. Now, since $2$ is adjacent to $1$ and $3$, we must have $g(2) = g(5)$; since $4$ is adjacent to $3$ and $5$, $g(4)=g(1)$; and since $6$ is adjacent to $1$ and $5$, $g(6) = g(3)$. This implies that the $3$-colorings $f$ and $g$ have the same configuration. Thus, up to permutations, there is only one $3$-coloring of $G$, and $P_G(3) = 3! = 6$.

Now we will compute $P_G(3)$ using Theorem \ref{Theorem 5.6}. It can be verified that the class of maximal independent sets is given by $\Omega = \left\{\{2,4,6\},\{1,4\},\{2,5\},\{3,6\}\right\}$. Denote by $l_i$ the monomial associated to the vertex $i$. Then 
\[\begin{array}{lll}
l_1 = x_{\{1\}}x_{\{1,4\}};& l_2 = x_{\{2\}}x_{\{2,4,6\}}x_{\{2,5\}}; &l_3 = x_{\{3\}}x_{\{3,6\}};\\
 l_4 = x_{\{4\}}x_{\{1,4\}}x_{\{2,4,6\}};& l_5=x_{\{5\}}x_{\{2,5\}};&l_6=x_{\{6\}}x_{\{2,4,6\}}x_{\{3,6\}}.
 \end{array}\]
 Define $m_1=l_4$ $m_2=l_5$, $m_3=l_6$ and $h_1 = l_1$,  $h_2=l_2$, $h_3 = l_3$. Then $M_G$ can be expressed in the form:\\
 $
M_G=(m_1 = x_{\{4\}}x_{\{1,4\}}x_{\{2,4,6\}}, \quad m_2 = x_{\{5\}}x_{\{2,5\}}, \quad m_3 =x_{\{6\}}x_{\{2,4,6\}}x_{\{3,6\}},$\\
 $h_1 = x_{\{1\}}x_{\{1,4\}},\quad
  h_2=x_{\{2\}}x_{\{2,4,6\}}x_{\{2,5\}},\quad h_3= x_{\{3\}}x_{\{3,6\}})$.\\
  By Theorem \ref{MonomialMult},
 \begin{align*}
 \e(S/M_G) &= \prod\limits_{i=1}^3 \deg(h_i) - \prod\limits_{i=1}^3 \deg\left(\dfrac{\lcm(m_1,h_i)}{m_1}\right) - \prod\limits_{i=1}^3 \deg\left(\dfrac{\lcm(m_2,h_i)}{m_2}\right)\\
 & - 
 \prod\limits_{i=1}^3 \deg\left(\dfrac{\lcm(m_3,h_i)}{m_3}\right) + \prod\limits_{i=1}^3 \deg\left(\dfrac{\lcm(m_1,m_2,h_i)}{\lcm(m_1,m_2)}\right) + 
 \prod\limits_{i=1}^3 \deg\left(\dfrac{\lcm(m_1,m_3,h_i)}{\lcm(m_1,m_3)}\right)\\
 &+ \prod\limits_{i=1}^3 \deg\left(\dfrac{\lcm(m_2,m_3,h_i)}{\lcm(m_2,m_3)}\right)- \prod\limits_{i=1}^3 \deg\left(\dfrac{\lcm(m_1,m_2,m_3,h_i)}{\lcm(m_1,m_2,m_3)}\right)\\
 &=(2\times 3\times2)-(1\times2\times2)-(2\times2\times2)-(2\times2\times1)+(1\times1\times2)+(1\times2\times1)+\\
 &+(2\times1\times1)-(1\times1\times1)=1
 \end{align*}
 Finally, by Theorem \ref{Theorem 5.6}, $P_G(3)=3! \e(S/M_G)=3!=6.$\\
 
 (ii) Consider the graph $G$ shown below.
 
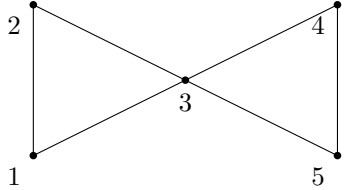
\[\begin{tikzpicture}
\draw (0,0)--(0,2);
\draw (0,2)--(4,0)--(4,2)--(0,0);
 \node[circle, label=below left:{5}, fill=black, inner sep=1pt] at (4,0) {};
\node[circle, label=below:{3}, fill=black, inner sep=1pt] at (intersection of 0,0--4,2 and 0,2--4,0) {};
  \node[circle, label=below left:{1}, fill=black, inner sep=1pt] at (0,0) {};
   \node[circle, label=below left:{2}, fill=black, inner sep=1pt] at (0,2) {};
        \node[circle, label=below left:{4}, fill=black, inner sep=1pt] at (4,2) {};
\end{tikzpicture}\]
It is easy to see that $\chi(G)=3$, and not much harder to verify that, up to permutations, there are two possible $3$-colorings of $G$:
\[\begin{array}{lll}
f(1) = f(4)=0,&f(2)=f(5)=1,&f(3)=2\text{, and}\\
g(1)=g(5)=0,&g(2)=g(4)=1,&g(3)=2.
\end{array}\]
Therefore, $P_G(3)=3!\times 2=12$. We can confirm this result using Theorem \ref{Theorem 5.6}. Notice that $\Omega=\left\{\{1,4\},\{1,5\},\{2,4\},\{2,5\},\{3\}\right\}$. Denote by $l_i$ the monomial associated to the vertex $i$. Then 
\[\begin{array}{lll}
l_1=x_{\{1\}}x_{\{1,4\}}x_{\{1,5\}},&l_2=x_{\{2\}}x_{\{2,4\}}x_{\{2,5\}},&l_3=x_{\{3\}}\\
l_4=x_{\{1\}}x_{\{1,4\}}x_{\{2,4\}},&l_5=x_{\{5\}}x_{\{1,5\}}x_{\{2,5\}}.&
\end{array}\]
Define $m_1=l_4$, $m_2=l_5$ and $h_1=l_1$, $h_2=l_2$, $h_3=l_3$. Then $M_G$ can be expressed in the form:\\
 $
M_G=(m_1 = x_{\{1\}}x_{\{1,4\}}x_{\{2,4\}}, \quad m_2 = x_{\{5\}}x_{\{1,5\}}x_{\{2,5\}}, \quad h_1 =x_{\{1\}}x_{\{1,4\}}x_{\{1,5\}},$\\
 $  h_2=x_{\{2\}}x_{\{2,4,\}}x_{\{2,5\}},\quad h_3= x_{\{3\}}x_{\{3,6\}})$.\\
 By Theorem \ref{MonomialMult},
  \begin{align*}
 \e(S/M_G) &= \prod\limits_{i=1}^3 \deg(h_i) - \prod\limits_{i=1}^3 \deg\left(\dfrac{\lcm(m_1,h_i)}{m_1}\right) - \prod\limits_{i=1}^3 \deg\left(\dfrac{\lcm(m_2,h_i)}{m_2}\right)\\
 & 
+  \prod\limits_{i=1}^3 \deg\left(\dfrac{\lcm(m_1,m_2,h_i)}{\lcm(m_1,m_2)}\right)\\
 &=(3\times 3\times1)-(2\times2\times1)-(2\times2\times1)+(1\times1\times1)=2.
 \end{align*}
 Finally, by Theorem \ref{MonomialMult}, $P_G(3)=3!\e(S/M_G)=3!2=12$.
\end{example}

\section{Final remarks}
 There are some problems on graph theory (such as the Erdős-Faber-Lovász conjecture), where computing chromatic numbers is difficult. In such cases, it is of interest to find upper bounds to these invariants (for upper bounds on the Erdős-Faber-Lovász conjecture, see [CL, Ka]). Given that monomial ideals are bounded above by their projective dimension, Theorem \ref{Theorem 3} provides upper bounds to chromatic numbers, namely, $\chi(G) = \codim(S/M_G) \leq \pd(S/M_G)$.
 
 Unfortunately, this upper bound is not sharp. Since $M_G$ is dominant (Proposition \ref{chrDOM}), $\pd(S/M_G)$ is equal to the number of minimal generators of $M_G$ [Al], which equals the number of vertices of $G$ (Definition \ref{DefChrom}). In other words, the inequality $\chi(G) \leq \pd(S/M_G)$ only states that the chromatic number is at most as large as the number of vertices, which is trivial.
 
 We close this article by proposing a line of investigation that may improve the upper bound given by $\pd(S/M_G)$. With the notation of Definition \ref{DefChrom}, suppose that we define the ideal $M'_G=(m'_1, \ldots, m'_n)$, where $m'_i = \prod\limits_{\substack{\omega\in \Omega\\ i\in \omega}} x_{\omega}$. Then, unless $G$ is an edgeless graph, $M'_G$ will not be dominant, and its projective dimension will be strictly less than the number of vertices of $G$. Is there a class of graphs $G$ for which $\pd(S/M'_G)$ is a sharp upper bound? When is $M'_G$ Cohen-Macaulay?

 \bigskip

\noindent \textbf{Acknowledgements}: After living in the U.S. for many years, my family and I had to return to our home country Argentina to comply with visa requirements. In the midst of much adversity, my parents in law prepared an old quiet farm for us to live, and my parents supported us financially. My dear wife Danisa, who has remained my closest friend through the years, typed this article, and our five children helped by doing their chores and schoolwork without complaint. This work was not supported by any grants, but it received great support from my loved ones, which I gratefully acknowledge.


\begin{thebibliography}{9}

\bibitem[Al]{alesandroni} G. Alesandroni, \emph{Minimal resolutions of dominant and semidominant ideals}, J. Pure Appl. Algebra {\bf 221} (2017), no. 4, 780-798.

\bibitem[Al1]{alesandroni} G. Alesandroni, \emph{Monomial multiplicities in explicit form}. To appear in J. Algebra Appl. (arXiv: 1901.03291v3). 

\bibitem[CL]{changlawler}  W. Chang, and E. Lawler, \emph{Edge coloring of hypergraphs and a conjecture of Erdős, Faber, Lovász}, Combinatorica  {\bf 8} (1988), 293-295.

\bibitem[Ei]{eisenbud} D. Eisenbud, \emph{Commutative Algebra with a View Toward Algebraic Geometry}, Springer-Verlag, New York-Berlin-Heidelberg, (1995).

\bibitem[Er]{erdos} P. Erdős, \emph{On the combinatorial problems that I would most like to see solved}, Combinatorica {\bf 1} (1981), 25-42.

\bibitem[FHVT]{franciscohavantuyl} C. Francisco, H. Hà, and A. VanTuyl, \emph{Colorings of hypergraphs, perfect graphs, and associated primes of powers of monomial ideals}, J. Algebra {\bf 331} (2011), no.1, 224-242.

\bibitem[Hi]{hindman} N. Hindman, \emph{On a conjecture of Erdős, Faber, and Lovász about $n$-colorings}, Can. J. Math. {\bf 33} (1981), 563-570.

\bibitem[Ka]{kahn} J. Kahn, \emph{Coloring nearly-disjoint hypergraphs with $n+o(n)$ colors}, J. Combin. Theory Ser. A {\bf 59} (1992), 31-39.

\bibitem[Pe]{Peeva} I. Peeva, \emph{Graded Syzygies}, Algebra and Applications, vol. 14, Springer, London (2010).

\bibitem[VT]{vantuyl} A. Van Tuyl, \emph{A beginner's guide to edge and cover ideals}, Monomial ideals, computations and applications, Lecture Notes Math., vol. 2083, Springer, Heidelberg (2013), 63-94.

\end{thebibliography}
\end{document}